\documentclass[11pt]{amsart}
\usepackage{amsmath,amsthm,amssymb,url,hyperref}
\usepackage[all]{xy}

\newtheorem*{theorems}{Theorem}
\newtheorem{theorem}{Theorem}[section]
\newtheorem{lemma}[theorem]{Lemma}

\theoremstyle{definition}
\newtheorem{definition}[theorem]{Definition}
\newtheorem{example}[theorem]{Example}

\theoremstyle{remark}
\newtheorem{remark}[theorem]{Remark}

\newcommand{\de}{\partial}

\newcommand{\sA}{\mathcal{A}}

\newcommand{\g}{\mathfrak{g}}
\newcommand{\m}{\mathfrak{m}}
\newcommand{\K}{\mathbb{K}}

\newcommand{\Id}{\operatorname{Id}}
\newcommand{\Spec}{\operatorname{Spec}}

\newcommand{\MC}{\operatorname{MC}}
\newcommand{\Def}{\operatorname{Def}}

\newcommand{\tot}{\operatorname{Tot}}
\newcommand{\Hilb}{\operatorname{Hilb}}
\newcommand{\Art}{\mathbf{Art}}
\newcommand{\Set}{\mathbf{Set}}

\newenvironment{acknowledgement}{\par\addvspace{17pt}\small\rm
\trivlist\item[\hskip\labelsep{\it Acknowledgement.}]}
{\endtrivlist\addvspace{6pt}}

\begin{document}

\title{Deformations of algebraic subvarieties}
\author{Donatella Iacono}
\address{\newline Max-Planck Institut f\"ur Mathematik,
\hfill\newline Vivatsgasse 7, D 53111 Bonn Germany}
\email{iacono@mpim-bonn.mpg.de}
\date{}

\begin{abstract}
In this paper, we use (bi)semicosimplicial language  to study the classical problem of infinitesimal deformations of a closed subscheme in a fixed smooth variety,  defined over an algebraically closed
field of characteristic 0. In particular, we give an explicit description of the  differential
graded Lie algebra controlling this problem.
\end{abstract}

\maketitle

\section*{Introduction}

In the last fifty  years,   deformation  theory has played an
important role in algebraic and complex geometry.
The main goal is the classification of  families of geometric
objects   in such a way that the classifying space  (the so
called moduli space) is a  reasonable geometric space. In
particular, each point of our moduli space corresponds to one
geometric object (class of isomorphism).
The study of small deformations of the complex structures of
complex manifolds started with the works of K.~Kodaira and
D.C.~Spencer \cite{bib Kodaira SpencerII} and   M.~Kuranishi
\cite{bib kuranishi}.
Then, A.~Grothendieck \cite{bib Grothendieck}, M.~Schlessinger
\cite{bib Schlessinger } and M.~Artin \cite{bib Artin} formalized
this theory translating it into a functorial language. The idea
is that, with   infinitesimal deformations of a geometric object,
we can associate a deformation functor of Artin rings  $F: \Art \to \Set$.  For example, we can study the functor $\Def_X$ of infinitesimal deformations of a variety $X$ or the functor $ \Hilb_{X}^Z$ of infinitesimal deformations of a subvariety $Z$ in a fixed variety $X$.
The fundamental fact is that, using these functors, we are able to study the formal neighborhood of a point in the moduli space. In particular, we can determine the tangent space or analyze the obstructions (smoothness) problem. 

A  modern approach to   the study of deformation functors, associated with geometric objects,  is via differential
graded Lie algebras   or, in general, via
$L_\infty$-algebras.  At this stage, we can think about these structures  as a generalization of   differential graded vector spaces
in which we   have    a bracket, plus some   compatibility
conditions between the differential and the bracket.
Once we have  a differential graded Lie algebra     $L$, we can define the associated deformation functor  $\Def_L: \Art \to \Set$,   using the solutions of the Maurer-Cartan equation up to gauge equivalence.

The guiding principle is the idea  due, at least, to P.~Deligne, V.~Drinfeld,
D.~Quillen and M.~Kontsevich  \cite{K} that \lq
\lq in characteristic zero  every deformation problem is
controlled by a differential graded Lie algebra".
In other words, if $F$ is the deformation functor  associated with a geometric problem, then there exists a differential graded Lie algebra $L$ (up to quasi-isomorphism) such that $\Def_L \cong F$.
We point out  that it is easier to study a deformation functor associated with a differential graded Lie algebra   but, in general, it
is not an easy task to find the right differential graded Lie algebra (up to
quasi-isomorphism) associated with the problem \cite{{konNOTE}}.
A first example, in which the associated differential graded Lie algebra is well understood,
is the case of deformations of  complex manifolds. If $X$ is a
complex compact manifold, then the infinitesimal deformations of $X$ are controlled by its Kodaira-Spencer algebra $KS_X$, see \cite{GoMil2,manRENDICONTi,ManettiSeattle} or \cite[Theorem II.7.3]{Iaconophd}. We recall that  $KS_X=\oplus_i \Gamma(X,\sA_X^{0,i}(\Theta_X))$, where $\sA_X^{0,i}(\Theta_X)$ is the sheaf of the  $(0,*)$-forms on $X$, with values in the holomorphic tangent bundle $\Theta_X$.

In general, if we work over an algebraically closed field   of characteristic zero, different from the complex numbers,  then we can not use the  Kodaira-Spencer algebra.

A  strategy to solve this problem, and \lq \lq produce\rq \rq differential graded Lie algebras,   is via semicosimplicial objects \cite{hinich,Pridham,scla,DDE coherent}.

\noindent Actually, the fundamental idea  goes back to K.~Kodaira and D.C.~Spencer:\lq\lq a deformation of $X$ is regarded as the gluing of the same polydisks
via different identifications\rq\rq \cite[pag.~182]{Kobook}.
In other words, a deformation of a  geometric object
consists in deforming the object locally and then glue
back together these local deformations.
Then, from the algebraic point of view, we have to  find the algebraic objects that control  locally the deformations and then glue them together.
Thus, we can think at a semicosimplcial object as a sequence of objects, that controls locally the deformations, and a sequence of maps, that controls the gluing. For example, let   $X$ be a smooth projective variety, over an algebraically closed field $\K$ of characteristic 0, with  tangent sheaf $\Theta_X$. Given an affine open cover $\mathcal{U}=\{U_i\}$ of
$X$,  we can define  the   \v{C}ech semicosimplicial Lie
algebra $\Theta_X(\mathcal{U})$, i.e.,  we have a sequence of Lie algebras $\{\mathfrak{g}_k= {\prod_{i_0<\cdots <i_k}\Theta_X(U_{i_0 \cdots  i_k})}\}$ and a \lq\lq lot\rq\rq of maps among them, that  are the restrictions to open subsets. In particular,  $ \mathfrak{g}_0= \prod_{i}\Theta_X(U_{i}) $ and  each $\Theta_X(U_{i})$ controls the infinitesimal deformations of $U_{i}$; moreover, the maps controls the gluing of deformations, see
 \cite{scla} and \cite[Section 5]{Iacono Manetti}.

In general, we will have a semicosimplicial differential graded Lie algebra, $\mathfrak{g}^\Delta=\{\mathfrak{g}_k\}_k$, with $\mathfrak{g}_0$ that controls the deformations of each open of the cover, as in the case of deformations of varieties or of coherent sheaves \cite{DDE coherent,dde2}.

Next, once we have a  semicosimplicial  differential graded Lie algebra $\mathfrak{g} ^\Delta$, we need to find out just one   differential graded Lie algebra. Following \cite{navarro,scla}, there is a canonical way to define a differential graded  Lie algebra
${\tot}_{TW}(\g^{\Delta})$, using the   Thom-Whitney construction.
In conclusion, given a geometric deformation problem, if we are able to associate with it a semicosimplicial  differential graded Lie algebra, then   we can  find out  just one differential graded Lie algebra   controlling our problem.

Inspired by these ideas, in this paper we use semicosimplicial language to study infinitesimal deformations of closed subschemes.
More precisely, let $X$ be a smooth variety, defined over an algebraically closed
field $\K$ of characteristic 0, and  $Z \subset X$   a closed subscheme.
Denote by $ \Hilb_{X}^Z$ the functor of infinitesimal deformations of
$Z$ in $X$ and  by ${\Hilb'}_{X}^Z$ the subfunctor of  locally trivial infinitesimal   deformations. We recall that  $ \Hilb_{X}^Z={\Hilb'}_{X}^Z$, whenever $Z$ is smooth.
For $\K= \mathbb{C}$ and $Z$ smooth, the analysis of this problem via differential graded Lie algebra is due to M. Manetti \cite{ManettiSemireg}.
Here, we   extend his work to all  algebraically closed
fields $\K$ of characteristic 0, using semicosimplicial language; more precisely,   it is convenient to use bisemicosimplial Lie algebras.
Indeed, let ${\Theta}_X$ be the tangent sheaf of $X$ and ${\Theta}_X(-\log Z)$ the sheaf of tangent vectors to $X$ which are tangent to $Z$. Denote by  ${\chi}  : {\Theta}_X(-\log Z) \hookrightarrow {\Theta}_X$  the inclusion of sheaves
of Lie algebras. We can associate with ${\Theta}_X(-\log Z) $ and $
{\Theta}_X$ the  \v{C}ech semicosimplicial Lie algebra ${\Theta}_X(-\log Z) (\mathcal{U})$ and ${\Theta}_X (\mathcal{U})$, respectively; and so we can consider the
bisemicosimplicial Lie algebra $\chi^\blacktriangle:{\Theta}_X(-\log Z) (\mathcal{U}) \to {\Theta}_X (\mathcal{U})$.
Once again,  using the  Thom-Whitney  construction, we can define a
differential graded Lie algebra ${\tot_{TW}^\blacktriangle(\chi^\blacktriangle)}$. This algebra controls the deformations of the closed subscheme $Z$;
more precisely, we prove the following theorem.

\begin{theorems}[A]
Let $X$ be a smooth variety, defined over an algebraically closed
field $\K$ of characteristic 0, and  $Z \subset X$   a closed subscheme. Then,
there exists an isomorphism  of functors
$\Def_{\tot_{TW}^\blacktriangle(\chi^\blacktriangle)}  \cong
{\Hilb'}_{X}^Z$. In particular, if $Z\subset X$ is smooth, then
$\Def_{\tot_{TW}^\blacktriangle(\chi^\blacktriangle)}   \cong
{\Hilb}_{X}^Z$.

\end{theorems}

In a forthcoming paper, we will use this theorem to study the obstruction to deformations of $Z$ in $X$, via the semiregularity map.

\medskip

\noindent The paper goes as follows: the first section is intended for the
nonexpert reader and is devoted to recall the basic notions of
differential graded Lie algebras  and their
role in  deformation theory.

\noindent In Section~\ref{sezio semicosimplicial},
we introduce  semicosimplicial objects and total constructions. In particular, we review
semicosimplicial differential graded Lie algebras,
the corresponding Thom-Whitney DGLA  and the associated deformation functors.

\noindent  Sections  \ref{sezio bisemicosi} is devoted to bisemicosimplicial objects and, again,  to the total constructions and the associated deformation functors. In particular, we describe the bisemicosimplicial Lie algebra $\chi^\blacktriangle: {\Theta}_X(-\log Z) (\mathcal{U}) \to {\Theta}_X (\mathcal{U})$, associated with
the inclusion  ${\chi}  : {\Theta}_X(-\log Z) \hookrightarrow {\Theta}_X$.

\noindent In Section \ref{section deformazioni Z in X}, we go back to geometric applications and we prove   Theorem  A.

\begin{acknowledgement}

I gratefully acknowledge  the Max Planck Institute of Bonn, where this note was carried out, for financial support and warm hospitality.
I wish to thank Marco Manetti for many useful discussions, advices
and suggestions.
This paper is dedicated  to   Marialuisa J. de Resmini:
I am   indebted  with her as she guides my first steps into the world of research.

\end{acknowledgement}

\noindent  {\bf Notation.} Throughout the paper, we work over an algebraically closed field $\K$ of characteristic zero.  All vector spaces, linear maps, tensor products etc. are intended over $\K$.
We denote by  $\Set$ the category of
sets (in a fixed universe) and by  $\Art$ the category of local Artinian $\K$-algebras (with residue field $\K$).
If $A$ is an object in $\Art$, then $\mathfrak{m}_A$ denotes its maximal ideal.

\bigskip

\section{Review of differential graded Lie algebras}\label{sezioDGLa l-infinito}

A differential graded vector space is a pair $(V,d)$, where
$V=\oplus_{i\in
\mathbb{Z}} V^i$ is a $\mathbb{Z}$-graded vector space and $d$ is a differential
of degree +1, i.e., $d:V^i \to V^{i+1}$ and $d \circ d=0$.
For every  integer $n$, we define a new differential graded
vector space $V[n]$, by setting
\[ V[n]^i=V^{n+i}\qquad \mbox{ and } \qquad  d_{V[n]}=(-1)^nd_V.\]

\begin{definition}\label{definizio DGLA}
A differential graded Lie algebra (DGLA for short) is a
triple $(L,[\, ,\, ],d)$, where $\displaystyle (L=\oplus_{i\in
\mathbb{Z}}L^i,d)$ is a differential graded vector space and  $[\, , \,]:
L\times L\to L$ is a bilinear map of degree zero, called bracket,  satisfying the following conditions:

 \begin{enumerate}

\item (graded  skewsymmetry) $[a,b]=-(-1)^{\deg(a)\deg(b)}[b,a]$;

\item (graded Jacobi identity) $
[a,[b,c]]=[[a,b],c]+(-1)^{\deg(a)\deg(b)}[b,[a,c]]$;

\item (graded Leibniz rule) $d[a,b]=[da,b]+(-1)^{\deg(a)}[a,db]$.
\end{enumerate}

\end{definition}

\begin{example}
If $L=\oplus L^i$ is a DGLA, then $L^0$ is a Lie algebra in the
usual sense; vice-versa, every Lie algebra is a differential
graded Lie algebra concentrated in degree 0 (and differential zero).

\end{example}

\begin{example}\label{exe L  DGLA LxA}
If $L$ is a DGLA and $B$ is a commutative $\K$-algebra, then $L
\otimes B$ has a natural structure of DGLA, given by
$$
[l\otimes a , m \otimes b]=[l,m]\otimes ab;
$$
$$
d(l\otimes a)=dl \otimes a.
$$
\end{example}

\smallskip

A morphism of differential graded Lie algebras $\phi\colon L \to M$ is  a linear map
that preserves degrees and commutes with brackets and
differentials. A quasi-isomorphism of DGLAs is a morphism
that induces an isomorphism in cohomology. Two DGLAs $L$ and $M$ are said to be
quasi-isomorphic if they are equivalent under the
equivalence relation generated by: $L\sim M$ if there exists
a quasi-isomorphism $\phi: L\to M$.\par

\subsection{Deformation functor associated with a DGLA}

\begin{definition}
Let $L$ be a DGLA; then, the  Maurer-Cartan functor associated
with $L$  is the functor
$$
MC_L:\Art \to \Set,
$$
$$
MC_L(A)=\{ x \in L^1\otimes m_A \ |\ dx+ \displaystyle\frac{1}{2}
[x,x]=0 \}.
$$
\end{definition}
\noindent Note that, in the previous equation, we use the DGLA structure on $L\otimes m_A$ induced by the one on $L$ (see Example
\ref{exe L  DGLA LxA}).

\begin{definition}\label{definizio gauge L DGLA}
Two elements $x$ and  $y \in L^1\otimes m_A$ are
 gauge equivalent if there exists $a \in L^0\otimes m_A$, such
that
$$
y=e^a * x:=x+\sum_{n\geq 0}  \frac{ [a,-]^n}{(n+1)!}([a,x]-da).
$$
\end{definition}
\noindent The operator $*$ is called the gauge action of the group $\exp(L^0
\otimes m_A)$ on $L \otimes m_A$; indeed, $e^a*e^b*x=e^{a \bullet
b}*x$, where $\bullet$  is the Baker-Campbell-Hausdorff product in
the nilpotent DGLA $L \otimes m_A$, i.e., $a \bullet b  =a+b +
\displaystyle\frac{1}{2}[a,b]+\frac{1}{12} [a,[a,b]]-\frac{1}{12}
[b,[b,a]]+ \cdots $.

\begin{definition}\label{dhef funtore defo DLGA DEF_L}
The deformation functor associated with a differential
graded Lie algebra  $L$ is
$$
\Def_L:\Art \to \Set,
$$
$$
\Def_L(A)=\frac{\MC_L(A)}{\text{gauge}}=\frac{\{ x \in L^1\otimes \mathfrak{m}_A \ |\ dx+
\displaystyle\frac{1}{2} [x,x]=0 \}}{\exp(L^0\otimes \mathfrak{m}_A )  }.
$$
\end{definition}

\begin{remark}
Every morphism of DGLAs induces a natural transformation of the
associated deformation functors. If
$L$ and $M$ are quasi-isomorphic DGLAs, then the associated
functor $\Def_L$ and $\Def_M$ are isomorphic
\cite{SchSta,GoMil1,GoMil2}, \cite[Corollary~3.2]{ManettiDGLA}, or
\cite[Corollary~5.52]{manRENDICONTi}.
\end{remark}

\section{Semicosimplicial  objects}\label{sezio semicosimplicial}

Let $\mathbf{\Delta}_{\operatorname{mon}}$ be  the category whose
objects are the finite ordinal sets $[n]\!=\!\{0,1,\ldots,n\}$,
$n=0,1,\ldots$, and whose morphisms are order-preserving injective
maps among them. Every morphism in
$\mathbf{\Delta}_{\operatorname{mon}}$, different from the
identity, is a finite  composition of coface morphisms:
\[
\partial_{k}\colon [i-1]\to [i],
\qquad \partial_{k}(p)=\begin{cases}p&\text{ if }p<k\\
p+1&\text{ if }k\le p\end{cases},\qquad k=0,\dots,i.
\]
The relations about compositions of them are generated by
\[ \partial_{l}\partial_{k}=
\partial_{k+1}\partial_{l}\, ,\qquad\text{for every }l\leq k.\]

\begin{definition}\label{def simplicial}
According to \cite{EZ,weibel}, a semicosimplicial object in
a category $\mathbf{C}$ is a  covariant functor $A^\Delta\colon
\mathbf{\Delta}_{\operatorname{mon}}\to \mathbf{C}$. Equivalently,
a semicosimplicial  object $A^\Delta$ is a diagram in
$\mathbf{C}$:
 \[
\xymatrix{ {A_0} \ar@<2pt>[r]\ar@<-2pt>[r] & { A_1}
      \ar@<4pt>[r] \ar[r] \ar@<-4pt>[r] & { A_2}
\ar@<6pt>[r] \ar@<2pt>[r] \ar@<-2pt>[r] \ar@<-6pt>[r]& \cdots ,}
\]
where each $A_i$ is in $\mathbf{C}$, and, for each $i>0$, there
are $i+1$ morphisms
\[
\partial_{k}\colon {A}_{i-1}\to {A}_{i},
\qquad k=0,\dots,i,
\]
such that $\partial_{l}
\partial_{k}=\partial_{k+1} \partial_{l}$, for any $l\leq k$.
\end{definition}

\begin{example}\label{exa morfismo come semicosim 0}
Let $\chi:L \to M$ be a morphism in a category  $\mathbf{C}$.
Then, we can consider it as a semicosimplicial object in
$\mathbf{C}$, by extension with zero, i.e.,
 \[\mathfrak{\chi}^\Delta:\quad
\xymatrix{ {L} \ar@<2pt>[r]\ar@<-2pt>[r] & { M}
      \ar@<4pt>[r] \ar[r] \ar@<-4pt>[r] & 0 \cdots },\qquad
      \partial_{0}=\chi \mbox{ and } \partial_{1}=0.
\]
\end{example}

\begin{example} \label{example. chec F fascio}
Let $X$ be a smooth variety, defined over an algebraically closed
field of characteristic 0. Let $\mathcal{U}=\{U_i\}$ be an affine
open cover and $\mathcal{F}$ a sheaf of Lie algebras on $X$. Then,
we can define  the   \v{C}ech semicosimplicial Lie algebra
$\mathcal{F}(\mathcal{U})$ as the semicosimplicial Lie algebra
\[
 \mathcal{F}(\mathcal{U}):\quad \xymatrix{ {\prod_i\mathcal{F}(U_i)}
\ar@<2pt>[r]\ar@<-2pt>[r] & { \prod_{i<j}\mathcal{F}(U_{ij})}
      \ar@<4pt>[r] \ar[r] \ar@<-4pt>[r] &
      {\prod_{i<j<k}\mathcal{F}(U_{ijk})}
\ar@<6pt>[r] \ar@<2pt>[r] \ar@<-2pt>[r] \ar@<-6pt>[r]& \cdots},
\]
where  the coface maps   $  \displaystyle \partial_{k}\colon
{\prod_{i_0<\cdots <i_{h-1}}\mathcal{F}(U_{i_0 \cdots
i_{h-1}})}\to {\prod_{i_0<\cdots <i_h}\mathcal{F}(U_{i_0 \cdots
i_h})}$ are given by
\[\partial_{k}(x)_{i_0 \ldots i_{h}}={x_{i_0 \ldots
\widehat{i_k} \ldots i_{h}}}_{|U_{i_0 \cdots  i_h}},\qquad
\text{for }k=0,\ldots, h.\]
\end{example}

\subsection{The total construction}\label{section total for DG}
Given a  semicosimplicial differential
graded vector space
\[V^{\Delta}:\quad
\xymatrix{ {V_0}
\ar@<2pt>[r]\ar@<-2pt>[r] & { V_1}
      \ar@<4pt>[r] \ar[r] \ar@<-4pt>[r] & { V_2}
\ar@<6pt>[r] \ar@<2pt>[r] \ar@<-2pt>[r] \ar@<-6pt>[r]& \cdots ,}\]
the graded vector space $\bigoplus_{n\ge 0}V_n[-n]$ has two
differentials, i.e.,
\[ d=\sum_{n}(-1)^nd_n,\qquad \text{where}\quad d_n
\text{ is the differential of } V_n,
\]
and
\[
\partial=\sum_{i}(-1)^i\partial_i,\qquad \text{where}
\quad \partial_i\text{ are the coface maps}.\]
More explicitly, if
$v\in V^i_n$, then the degree of $v$ is $i+n$ and
\[ d(v)=(-1)^nd_n(v)\in V^{i+1}_n,\qquad
\partial(v)=\partial_0(v)-\partial_1(v)+\cdots+(-1)^{n+1}
\partial_{n+1}(v)\in V_{n+1}^i.\]
Since $d\partial+\partial d=0$, we define $\tot(V^{\Delta})$ as
the  graded vector space  $\bigoplus_{n\ge 0}V_n[-n]$, endowed
with the differential $D=d+\partial$.
\begin{remark}
In Example \ref{example. chec F fascio}, the  total complex $\tot(\mathcal{F}(\mathcal{U}))$,   associated with
 the  \v{C}ech semicosimplicial Lie algebra
$\mathcal{F}(\mathcal{U})$,
is nothing else that  the \v{C}ech
complex $\check{C}(\mathcal{U},\mathcal{F})$ of the sheaf $\mathcal{F}$.
\end{remark}
There is also another way to associate with   a  semicosimplicial differential graded vector space $V^{\Delta}$
 a differential graded vector space.
Namely, let   $(A_{PL})_n$ be the differential graded commutative algebra
of polynomial differential forms on the standard $n$-simplex
$\{(t_0,\ldots,t_n)\in \K^{n+1}\mid \sum t_i=1\}$ \cite{FHT}:
\[ (A_{PL})_n=\frac{\K[t_0,\ldots,t_n,dt_0,\ldots,dt_n]}
{(1-\sum t_i,\sum dt_i)}.\] For every $n,m$ the tensor product
$V_n\otimes (A_{PL})_m$ is a differential graded vector space and
then also $\prod_n V_n \otimes (A_{PL})_n$ is a differential
graded vector space.
Denote by
\[
\delta^{k}\colon (A_{PL})_n\to
(A_{PL})_{n-1},\quad
\delta^{k}(t_i)=\begin{cases}t_i&\text{ if }0\le i<k\\
0&\text{ if }i=k\\
t_{i-1}&\text{ if }k<i\end{cases},\qquad k=0,\dots,n,
\]
the face maps, for every $0\le k\le n$; then, there are well-defined
morphisms of   differential graded vector spaces
\[
Id\otimes \delta^{k}\colon V_n\otimes (A_{PL})_{n}\to V_n\otimes
(A_{PL})_{n-1},\]
\[\partial_{k}\otimes Id \colon V_{n-1}\otimes
(A_{PL})_{n-1}\to V_{n}\otimes (A_{PL})_{n-1}.
\]
The Thom-Whitney differential graded vector space
$\tot_{TW}(V^\Delta)$ of $V^\Delta$ is the  differential graded
subvector space of $\prod_n V_n \otimes (A_{PL})_n$, whose
elements are the sequences $(x_n)_{n\in\mathbb{N}}$ satisfying the
equations
\[(Id\otimes \delta^{k})x_n=
(\partial_{k}\otimes Id)x_{n-1},\; \text{ for every }\; 0\le k\le
n.
\]
\begin{lemma}
The differential graded vector spaces $\tot(V^{\Delta})$ and  $\tot_{TW}(V^\Delta)$ are quasi-isomorphic.
\end{lemma}
\begin{proof}
See \cite{whitney,dupont1,dupont2,navarro,getzler,scla,chenggetzler}  for explicit description of the quasi-isomorphism.
\end{proof}

\smallskip

Let
 \[\mathfrak{g}^\Delta:\quad
\xymatrix{ {{\mathfrak g}_0} \ar@<2pt>[r]\ar@<-2pt>[r] & {
{\mathfrak g}_1}
      \ar@<4pt>[r] \ar[r] \ar@<-4pt>[r] & { {\mathfrak g}_2}
\ar@<6pt>[r] \ar@<2pt>[r] \ar@<-2pt>[r] \ar@<-6pt>[r]& \cdots ,}
\]
be a semicosimplicial differential graded Lie algebra.
Since, every DGLA is, in particular, a differential graded vector space, we can consider the associated total complex $\tot(\mathfrak{g}^\Delta)$. Even if all $\mathfrak{g}_i$  are DGLAs, there is no natural DGLA  structure  on $\tot(\mathfrak{g}^\Delta)$ \cite{cone,Iacono Manetti}.
\begin{example}
Let $\chi:L \to M$ be a morphism of DGLAs, then,  following
Example \ref{exa morfismo come semicosim 0}, we can associate with
it a semicosimplicial DGLA. Its total complex
$\tot(\chi^{\Delta})$ is nothing else than the (suspension of the)
mapping cone complex associated with $\chi$. Even in this simple
case, it is not possible to define a canonical DGLA structure on
  $\tot(\chi^{\Delta})$, such that the projection
$\tot(\chi^{\Delta})\to L$ is  a morphism of DGLAs \cite[Example
3.1]{Iacono Manetti}.
\end{example}

\noindent However, in the case of semicosimplicial DGLAs, we can apply the Thom-Whitney construction to $g^{\Delta}$: it turns out that $\tot_{TW}(\mathfrak{g}^{\Delta})$ has a structure of DGLA \cite{navarro,scla}.

\begin{remark}
Using the homotopy transfer, the DGLA structure of $\tot_{TW}(\mathfrak{g}^{\Delta})$ induces an $L_\infty$-algebra structure $\widetilde{\tot}(\mathfrak{g}^{\Delta})$ on the differential
graded vector space ${\tot}(\mathfrak{g}^{\Delta})$, such that
$
\widetilde{\tot}(\mathfrak{g}^{\Delta})$ and $
{\tot}_{TW}(\mathfrak{g}^\Delta)$ are quasi-isomorphic;  see \cite{cone,scla} or \cite[Corollary 3.3]{Iacono Manetti}.
\end{remark}

\subsection{Deformation functor associated with semicosimplicial DGLAs}\label{section. deformat semico DGLA}

Let ${\g}^\Delta$ be a  semicosimplicial DGLA. Applying  the    Thom-Whitney construction of the previous section, we  can consider the  DGLA ${\tot}_{TW}(\g^{\Delta})$  and so it is well defined the associated   deformation functor $\Def_{{\tot}_{TW}(\g^{\Delta})} $.
Beyond this way, there is another natural, and more geometric, way
to  define a deformation functor associated with ${\g}^\Delta$, see   \cite[Definitions~1.4 and
1.6]{Pridham}, \cite[Section~3]{scla} or \cite[Definitions~2.1 and 2.2]{DDE coherent}.

More precisely, if ${
\g}^\Delta$ is a semicosimplicial DGLA,  we can define the functor
\[
 Z^1_{\rm sc}( \exp { \g}^\Delta ) \colon \Art \to \Set,
\]
such that, for all $A \in \Art$, $Z^1_{\rm sc}(\exp \g^{\Delta})(A)$ is the set of the pairs $(l,m)\in
(\g_0^1\otimes \mathfrak m_A) \oplus (\g^0_1 \otimes \mathfrak m_A) $, satisfying the following conditions:
\begin{enumerate}
  \item $dl+\frac{1}{2}[l,l]=0$;
  \item $\partial_{1}l=e^{m}*\partial_{0}l$;
  \item ${\partial_{0}m} \bullet {-\partial_{1}m} \bullet
{\partial_{2}m} =dn+[\partial_{2}\partial_{0}l,n]$, for some $n\in {\mathfrak g}_2^{-1}\otimes{\mathfrak m}_A$.
\end{enumerate}
Moreover, we define the functor
\[
 H^1_{\rm sc}( \exp { \g}^\Delta) \colon \Art \to \Set,
\]
such that
\[
H^1_{\rm sc}(\exp \g^\Delta)(A)= \frac{Z^1_{\rm sc}
(\exp \g^\Delta)(A)}{\sim},
\]
where
$(l_0,m_0)$ and $(l_1,m_1) \in Z^1_{\rm sc} (\exp \g^{\Delta})(A)$
are equivalent under the relation  $\sim$ if  and only if there
exist elements $a \in \g^0_0\otimes \m_A$ and $b\in{\mathfrak
g}_1^{-1}\otimes{\mathfrak m}_A$, such that
 \begin{enumerate}
   \item $ e^a * l_0=l_1$;
   \item $- m_0\bullet -\partial_{1}a \bullet m_1
\bullet \partial_{0}a=db+[\partial_{0}l_0,b]$.
 \end{enumerate}

\begin{example}
Let $L$ be a differential graded Lie algebra, then it can be considered as a semicosimplicial DGLA $\mathfrak {L}^\Delta$   by zero extension, i.e., $\mathfrak {L}^\Delta_0=L$ and $\mathfrak {L}^\Delta_i=0$, for all $i>0$. In this case, the above functors  $Z^1_{\rm sc}( \exp { \mathfrak {L}}^\Delta )$ and $H^1_{\rm sc}(\exp \mathfrak {L}^\Delta)$ reduce to $\MC_L$ and $\Def_L$, respectively.

\end{example}

\begin{example}\label{example.funtore morfismo DGLA}
If  $\chi: L \to M$ is a morphism of DGLAs, then we can consider it as a simple case of semicosimplicial DGLA ${\chi}^\Delta$, extending $\chi $ by zero (see Example
\ref{exa morfismo come semicosim 0}).

In this case, the functors  $Z^1_{\rm sc}( \exp { \chi}^\Delta)$  and $H^1_{\rm sc}( \exp { \chi}^\Delta)$ coincide with the  functors $\MC_{\chi}$ and $\Def_\chi$ defined in \cite[Section 2]{ManettiSemireg}. More precisely, we have
\[ \Def_\chi(A)=\frac{\MC_{\chi}(A)}{\exp(L^0\otimes\mathfrak{m}_A)
\times \exp(dM^{-1}\otimes\mathfrak{m}_A)},\]
where
\[ \MC_{\chi}(A)=
\left\{(x,e^a)\in (L^1\otimes\mathfrak{m}_A)\times \exp(M^0\otimes
\mathfrak{m}_A)\mid dx+\frac{1}{2}[x,x]=0,\;e^a\ast\chi(x)=0\right\},\]
and the gauge action of ${\exp(L^0\otimes\mathfrak{m}_A)
\times \exp(dM^{-1}\otimes\mathfrak{m}_A)}$ is given by the formula
\[ (e^l, e^{dm})\ast (x,e^a)=(e^l\ast x, e^{dm}e^ae^{-\chi(l)})=(e^l\ast x,
e^{dm\bullet a\bullet -\chi(l)}).\]%

In particular, if $\chi:L\to M$ is an injective morphism of DGLAs, then for every $A \in \Art$, we have
\[ \MC_{\chi}(A)=
\left\{e^a\in   \exp(M^0\otimes
\mathfrak{m}_A)\mid  \;e^{-a}\ast0 \in  L^1\otimes\mathfrak{m}_A  \right\}.\]
Under this identification, the gauge action becomes
\[\exp(L^0\otimes\mathfrak{m}_A)
\times \MC_{\chi}(A) \to \MC_{\chi}(A), \qquad (e^m,e^a)\mapsto e^ae^{-m}, \]
and then
\[ \Def_\chi(A)=\frac{\MC_{\chi}(A)}{\exp(L^0\otimes\mathfrak{m}_A).}
 \]
\end{example}

\begin{example}
If all $\g_i=0$, for all $i>1$, then the functors   $Z^1_{\rm sc}(\exp \g^\Delta)$ and $H^1_{\rm sc}(\exp \g^\Delta)$ reduce  to the  functors $\MC_{(\partial_0, \partial_1)}$ and  $\Def_{(\partial_0, \partial_1)}$, respectively,  associated with the pair of morphisms of DGLAs $\partial_0, \partial_1: \g_0 \to \g_1$, introduced in \cite{IaconoIMNR}.
\end{example}

\begin{example}\label{example.funtore se DGLA sono LIE}
If each $\g_i$ is concentrated in degree zero, i.e.,
  $\g^\Delta$ is a semicosimplicial  Lie
algebra,  then the functors   $Z^1_{\rm sc}(\exp \g^\Delta)$ and
$H^1_{\rm sc}(\exp \g^\Delta)$ reduce to
the one defined in \cite[Section 3]{scla}. More explicitly,
in this case, we have
\[
Z ^1_{sc}(\exp { \g}^\Delta )(A)=\{ x \in {\g}_1 \otimes
\mathfrak{m}_A   \ |\  e^{\de_{0}x}e^{-\de_{1}x}e^{\de_{2}x}=1
\},
\]
and $x \sim y$ if and only if there exists
$a\in {\g}_0\otimes\mathfrak{m}_A$, such that
$e^{-\de_{1}a}e^{x}e^{\de_{0}a}=e^y$.

\end{example}

Therefore, given a   semicosimplicial DGLA ${ \g}^\Delta$, we can define two deformation functors, $\Def_{{\tot}_{TW}(\g^{\Delta})}$ and $ H^1_{\rm sc}(\exp\g^\Delta)$.
The relation between these functors is given by the  following theorem.
\begin{theorem}\label{teor.semisipl H^1=Def TOT}
Let  ${ \g}^\Delta$ be  a semicosimplicial DGLA such that $H^{k}(\g_i)=0$, for all $i$ and for all
$k<0$. Then, there exists a
natural isomorphism of deformation functors
\[
\Def_{{\tot}_{TW}(\g^{\Delta})} \simeq
  H^1_{\rm sc}(\exp\g^\Delta).
\]
\end{theorem}

\begin{proof}
In the case of semicosimplicial Lie algebra, this theorem was proved in \cite[Theorem 6.8]{scla}. For the general case, see
\cite[Theorem 7.6]{DDE coherent}.
\end{proof}

\section{Bisemicosimplicial  objects}\label{sezio bisemicosi}

In this section, we generalize the notion of semicosimplicial
objects, defining bisemicosimplicial objects.

\begin{definition}\label{def-BI simplicial}
According to  \cite[Chapter IV]{goerss-jardine}, a
bisemicosimplicial object $A^ \blacktriangle$ in a
category $\mathbf{C}$ is a covariant functor $A^ \blacktriangle
\colon \mathbf{\Delta}_{\operatorname{mon}} \times
\mathbf{\Delta}_{\operatorname{mon}} \to \mathbf{C}$;
equivalently, a bisemicosimplicial object in $\mathbf{C}$
is a semicosimplicial object in the category of semicosimplicial
object in $\mathbf{C}$. More explicitly, it consists of objects
$A_{i,j}$, for all $i,j \geq 0$, and  morphisms $\partial^{V_i}_{k}$ and    $\partial^{H_j}_{s}$   in $\mathbf{C}$, for each $i,j>0$, such that
\[
\partial^{V_i}_{k}\colon {A}_{i,j-1}\to {A}_{i,j},
\qquad k=0,\dots,j,
\]
\[
\partial^{H_j}_{s}\colon {A}_{i-1,j}\to {A}_{i,j},
\qquad s=0,\dots,i,
\]
and the following compatibility conditions are satisfied

$$
\partial^{V_i}_{l} \circ \partial^{V_i}_{k}
 =\ \partial^{V_i}_{k+1}  \circ \partial^{V_i}_{l} , \qquad \forall \,
 l\leq k, \qquad
$$
$$
 \partial^{H_j}_{s}  \circ\partial^{H_j}_{ t}
 =\partial^{H_j}_{t+1} \circ \partial^{H_j}_{s},  \qquad  \forall  \,
  s\leq t, \qquad
$$
$$
\qquad \qquad \quad  { \partial^{H_{j+1}}_{s} } \circ {\partial^{V_i}_{k}} =
{\partial^{V_{i+1}}_{k}} \circ { \partial^{H_j}_{s} },
 \quad \forall  \,  s \leq i+1,\, \forall  \,  k \leq j+1.
$$
\end{definition}

\noindent We shall say that the object $A_{i,j}$ has bidegree $(i,j)$ or,
precisely, horizontal degree $i$ and vertical degree $j$, and that
$ { \partial^{H_j}_{\bullet } }$ and $ {\partial^{V_i}_{\bullet }}$ are horizontal (height $j$)
and vertical (column $i$) morphisms, respectively.
In particular, for all fixed $j$,
$(A_{\bullet,j},{\partial^{H_j}_{\bullet}} )$ is a
(horizontal) semicosimplicial object in $\mathbf{C}$; analogously, for
all fixed $i$, $(A_{i,\bullet}
,{\partial^{V_i}_{\bullet }} )$ is a (vertical) semicosimplicial
object in $\mathbf{C}$.
To sum up, a bisemicosimplicial  object $A^  \blacktriangle$
looks like a diagram
 \[
\xymatrix{   \vdots \  \ar@<2pt>[r]\ar@<-2pt>[r] & \vdots
      \ar@<4pt>[r] \ar[r] \ar@<-4pt>[r] & \vdots
\ar@<6pt>[r] \ar@<2pt>[r] \ar@<-2pt>[r] \ar@<-6pt>[r]& \cdots \\
{A_{0,2}} \ar@<2pt>[r]\ar@<-2pt>[r] \ar@<6pt>[u] \ar@<2pt>[u]
\ar@<-2pt>[u] \ar@<-6pt>[u]
  &
{A_{1,2}} \ar@<4pt>[r] \ar[r] \ar@<-4pt>[r] \ar@<6pt>[u]
\ar@<2pt>[u] \ar@<-2pt>[u] \ar@<-6pt>[u]
  &
{A_{2,2}} \ar@<6pt>[r] \ar@<2pt>[r] \ar@<-2pt>[r] \ar@<-6pt>[r]
\ar@<6pt>[u] \ar@<2pt>[u] \ar@<-2pt>[u] \ar@<-6pt>[u]
 & \cdots \\
{A_{0,1}} \ar@<2pt>[r]\ar@<-2pt>[r] \ar@<4pt>[u] \ar[u]
\ar@<-4pt>[u]
  &
{A_{1,1}} \ar@<4pt>[r] \ar[r] \ar@<-4pt>[r] \ar@<4pt>[u] \ar[u]
\ar@<-4pt>[u]
  &
{A_{2,1}} \ar@<6pt>[r] \ar@<2pt>[r] \ar@<-2pt>[r] \ar@<-6pt>[r]
\ar@<4pt>[u] \ar[u] \ar@<-4pt>[u]
  & \cdots \\
{A_{0,0}} \ar@<2pt>[r]\ar@<-2pt>[r] \ar@<2pt>[u]\ar@<-2pt>[u]
  &
{A_{1,0}} \ar@<4pt>[r] \ar[r] \ar@<-4pt>[r] \ar@<4pt>[r]
\ar@<2pt>[u]\ar@<-2pt>[u]
  &
{A_{2,0}} \ar@<6pt>[r] \ar@<2pt>[r] \ar@<-2pt>[r] \ar@<-6pt>[r]
\ar@<6pt>[r] \ar@<2pt>[u] \ar@<-2pt>[u]
  & \cdots ,}
\]
where each row and each column is a semicosimplicial object   and
each square commutes in a simplicial sense, i.e., for all $s,k,i$
and $j$, the following diagram commutes
 \[
\xymatrix{ {A_{i,j+1}}  \ar [r]^{ \partial^{H_{j+1}}_{s} }  &
 {A_{i+1,j+1}}  \\
{A_{i,j}} \ar [r]_{ \partial^{H_j}_{s} }
\ar[u]^{\partial^{V_i}_{k}}
  & {A_{i+1,j}}   \ar[u]_{\partial^{V_{i+1}}_{k}}.
 }
\]
\begin{example}

Every semicosimplicial object in a category $\mathbf{C}$ can be considered as a bisemicosimplcial object concentrated in zero (vertical or horizontal) degree.
\end{example}

Bisemicosimplicial objects naturally arise in simple situation. Indeed,
let $\mathcal{F}$ and $\mathcal{G}$   be  sheaves  on a variety $X$,
with value in a category $\mathbf{C}$. Fix an affine open cover  $\mathcal{U}=\{U_i\}$.  Then, as in Example
\ref{example. chec F fascio}, we denote by $\mathcal{F}(\mathcal{U})$ and $\mathcal{G}(\mathcal{U})$  the associated \v{C}ech semicosimplicial objects in $\mathbf{C}$.
 Next, let $\varphi :  \mathcal{F} \to \mathcal{G}$
be a morphism of sheaves. Since $\varphi $ commutes with
restrictions of every open subsets, it induces  a morphism
$\varphi^{\Delta} : \mathcal{F}(\mathcal{U}) \to
\mathcal{G}(\mathcal{U})$ of semicosimplcial objects. Finally, as in
Example \ref{exa morfismo come semicosim 0}, we can consider the
semicosimplicial extension  of $\varphi^{\Delta}$ (by zero) to get
a bisemicosimplcial object $\varphi^\blacktriangle : \mathcal{F}(\mathcal{U}) \to
\mathcal{G}(\mathcal{U})$ in $\mathbf{C}$.
This construction is commutative, i.e., we can firstly extend $\varphi $ (by zero) to get a semcosimplicial sheaf of object in $\mathbf{C}$, and then apply the  \v{C}ech
semicosimplicial  construction to all sheaves.

\begin{example}\label{example. chec TY(X)--TY}

Let $X$ be a smooth variety, defined over an algebraically closed
field of characteristic 0, and $\mathcal{U}=\{U_i\}$ be an affine
open cover. Let $Z \subset X$ be a closed subscheme of $X$.
We denote by ${\Theta}_X(-\log Z)$ the sheaf of germs of tangent vectors to $X$ which are tangent to $Z$ \cite[Section 3.4.4]{Sernesi}.
We recall that, if $\mathcal{I}\subset \mathcal{O_X}$ is the ideal sheaf of Z in X, then $\Theta(-\log  Z)=\{f\in Der(\mathcal{O_X},\mathcal{O_X})\mid f(\mathcal{I})
\subset\mathcal{I}\}$.
 Let   ${\chi}  : {\Theta}_X(-\log Z) \hookrightarrow
{\Theta}_X$ be the inclusion of sheaves
of Lie algebras.
Then, we can associate with ${\Theta}_X(-\log Z) $ and $
{\Theta}_X$ the  \v{C}ech semicosimplicial Lie algebra ${\Theta}_X(-\log Z) (\mathcal{U})$ and ${\Theta}_X (\mathcal{U})$, respectively. Finally, extending the morphism $\chi$ by zero, we get a a bisemicosimplicial Lie algebra $\chi^\blacktriangle: {\Theta}_X(-\log Z) (\mathcal{U}) \to {\Theta}_X (\mathcal{U})$.

\noindent More explicitly, we have the following  diagram
\[
\xymatrix{    \vdots \ \ar@<2pt>[r]\ar@<-2pt>[r] & \vdots
      \ar@<4pt>[r] \ar[r] \ar@<-4pt>[r] & 0 \\
\prod_{i<j<k} {\Theta}_X(-\log Z) (U_{ijk})
\ar@<2pt>[r]\ar@<-2pt>[r] \ar@<6pt>[u] \ar@<2pt>[u] \ar@<-2pt>[u]
\ar@<-6pt>[u]
  &
\prod_{i<j<k} {\Theta}_X  (U_{ijk})  \ar@<4pt>[r] \ar[r]
\ar@<-4pt>[r] \ar@<6pt>[u] \ar@<2pt>[u] \ar@<-2pt>[u]
\ar@<-6pt>[u]
  &
0 \ar@<6pt>[u] \ar@<2pt>[u] \ar@<-2pt>[u] \ar@<-6pt>[u]
   \\
\prod_{i<j}{\Theta}_X(-\log Z) (U_{ij})  \ar@<2pt>[r]\ar@<-2pt>[r]
\ar@<4pt>[u] \ar[u] \ar@<-4pt>[u]
  &
\prod_{i<j} {\Theta}_X (U_{ij})  \ar@<4pt>[r] \ar[r] \ar@<-4pt>[r]
\ar@<4pt>[u] \ar[u] \ar@<-4pt>[u]
  &
0 \ar@<4pt>[u] \ar[u] \ar@<-4pt>[u]
   \\
\prod_{i}{\Theta}_X(-\log Z) (U_{i})  \ar@<2pt>[r]\ar@<-2pt>[r]
\ar@<2pt>[u]\ar@<-2pt>[u]
  &
\prod_{i}{\Theta}_X (U_{i})  \ar@<4pt>[r] \ar[r] \ar@<-4pt>[r]
\ar@<4pt>[r] \ar@<2pt>[u]\ar@<-2pt>[u]
  &
0.   \ar@<2pt>[u] \ar@<-2pt>[u]
   ,}
\]

\end{example}

\subsection{The total construction}

Let $V^\blacktriangle= (V_{n,m}^*, d_{n,m}) _{n,m}$ be a bisemicosimplicial differential
graded vector space; in particular, we recall that each row and each column is a semicosimplicial differential
graded vector space.
Then, as in Section \ref{section total for DG},  with each horizontal semicosimplicial
differential graded vector space $(V_{\bullet,m}^\Delta,{\partial^{H_m}_{\bullet }} )$, we can  associate the total complex $\tot (V_{\bullet, m}^\Delta)$.
We recall that $\displaystyle \tot  (V_{\bullet, m}^\Delta)=\bigoplus _{n\geq0}V_{n,m}^*[-n]$ and its differential is $\displaystyle D_m= \sum_n (-1)^n d_{n,m} +\sum_j (-1)^j \partial^{H_m}_{j} $.
In this way, we construct a semicosimplicial differential graded vector space  $\tot^{H,\Delta}(V^\blacktriangle)$
\[
\xymatrix{    \vdots  \\
\tot  (V_{\bullet, 2}^\Delta) \ar@<6pt>[u] \ar@<2pt>[u]
\ar@<-2pt>[u] \ar@<-6pt>[u]  \\
\tot (V_{\bullet, 1}^\Delta) \ar@<4pt>[u] \ar[u] \ar@<-4pt>[u]
 \\
\tot (V_{\bullet, 0}^\Delta).  \ar@<2pt>[u]\ar@<-2pt>[u]
  ,}
\]

In particular, we can still apply the total construction to $\tot^{H,\Delta}(V^\blacktriangle)$ to obtain the differential graded vector space  $\tot(\tot^{H,\Delta}(V^\blacktriangle))$. More explicitly, $\displaystyle \tot(\tot^{H,\Delta}(V^\blacktriangle))= \bigoplus_m \tot  (V_{\bullet, m}^\Delta)[-m]= \bigoplus_{n,m} V_{n,m}^*[-n-m]$ and the differential is $D=\sum_m (-1)^m D_m +\sum_k (-1)^k\partial^{V_{\bullet}}_{k} =  \sum_{m,n} (-1)^{m+n} d_{n,m} +\sum_{j,m} (-1)^{j+m} \partial^{H_m}_{j}   +\sum_k (-1)^k\partial^{V_{\bullet}}_{k} $.

\medskip

Analogously, given $V^\blacktriangle=(V_{n,m}^*, d_{n,m})_{n,m}$, we can firstly focus our attention on each vertical semicosimplicial
differential graded vector space   $(V_{n,\bullet} ^\Delta
,{\partial^{V_n}_{\bullet}} )$. As before, we can associate with each column its total complex, to get a semicosimplicial differential graded vector space
  $\tot^{V,\Delta}(V^\blacktriangle)$
 \[
\xymatrix{ \tot(V_{0, \bullet }^\Delta)  \ar@<2pt>[r]
\ar@<-2pt>[r]  &  \tot(V_{1, \bullet }^\Delta)
      \ar@<4pt>[r] \ar[r] \ar@<-4pt>[r] &
       \tot(V_{2, \bullet }^\Delta)
\ar@<6pt>[r] \ar@<2pt>[r] \ar@<-2pt>[r] \ar@<-6pt>[r]& \cdots ,}
\]
In this case, $ \displaystyle \tot(V_{n, \bullet }^\Delta)=\bigoplus_m V_{n,m}^*[-m]$ and its differential is given by
${D'}_n= \sum_m (-1)^m d_{n,m} +\sum_j (-1)^j \partial^{V_n}_{j }$.
Then, applying again the total construction to   $\tot^{V,\Delta}(V^\blacktriangle)$, we get the differential graded vector space  $\tot(\tot^{V,\Delta}(V^\blacktriangle))$. In this case, we have $ \displaystyle \tot(\tot^{V,\Delta}(V^\blacktriangle)) = \bigoplus_n \tot(V_{n, \bullet }^\Delta)[-n]= \bigoplus_{n,m} V_{n,m}^*[-n-m]$ and the differential is $D'=\sum_n (-1)^n {D'}_n +\sum_k (-1)^k\partial^{H_{\bullet}}_{k  }= \sum_{m,n} (-1)^{n+m} d_{n,m} +\sum_{j,n} (-1)^{j+n} \partial^{V_n}_{j }+\sum_k (-1)^k\partial^{H_{\bullet}}_{k  }$.

Moreover, we can also consider the total complex  $(\tot^\blacktriangle (V^\blacktriangle),D)$ associated with the triple complex $(V_{n,m}^*,d_{n,m},\partial^V, \partial^H)  $. More explicitly,
$\tot^\blacktriangle (V^\blacktriangle)^i= \bigoplus_{n,m}V_{n,m}[-n-m]^{i-n-m}$
and  the differential is given by $D=d+\partial_1 +\partial_2$, where $d=\sum_{m,n} (-1)^{m+n} d_{n,m} $, $\partial_1= \sum_{j,m} (-1)^{j+m} \partial^{H_m}_{j }$ and  $\partial_2 = \sum_k (-1)^{k} \partial^{V_{\bullet}}_{k }$.

\begin{lemma}
Let $V^\blacktriangle=(V_{n,m}^*, d_{n,m})_{n,m}$ be a bisemicosimplicial differential graded vector space.
Then, the associated differential graded vector spaces $(\tot^\blacktriangle (V^\blacktriangle),D)$, $\tot(\tot^{H,\Delta}(V^\blacktriangle))$ and  $\tot(\tot^{V,\Delta}(V^\blacktriangle))$
are quasi isomorphic.
\end{lemma}

\begin{proof}
It follows from  a standard computation, using spectral sequence.

\end{proof}

\bigskip

As in the previous section, we can also apply the  Thom-Whitney construction instead of the total complex construction.
Also in this case, we get two differential graded  vector spaces
$\tot_{TW}(\tot_{TW}^{H,\Delta} )$ and $\tot_{TW}(\tot_{TW}^{ \Delta, V} )$ depending, a priori, on the order of the construction.
There is also a more direct way, based on the Thom-Whitney construction,  to associate a differential graded vector space with a bisemicosimplicial differential graded vector space.

\begin{definition}
Let $V^  \blacktriangle=(V_{n,m})$ be a bisemicosimplicial DGLA.
The Thom-Whitney DGLA $\tot_{TW}^\blacktriangle(V^\blacktriangle)$
is defined as the sub-differential graded vector space of
$\prod_{n,m} V_{n,m} \otimes (A_{PL})_n \otimes (A_{PL})_m$, whose
elements are sequences $(x_{n,m})_{n,m}$ satisfying the relations:
\[
(\partial^{H_m}_{k }\otimes Id \otimes Id)x_{n,m}= (Id \otimes \delta^{k} \otimes
 Id ) x_{n+1,m},\; \text{ for every }\; 0\le k\le
n,
\]
and
\[
(\partial^{V_n}_{k }\otimes Id \otimes Id)x_{n,m}= (Id \otimes
 Id \otimes \delta^{k}) x_{n,m+1} ,\; \text{ for every }\; 0\le k\le
m.
\]
More explicitly, we are considering sequence of elements $(x_{n,m})_{n,m}=x_{n,m} \otimes \alpha_n \otimes
\beta_m \in V_{n,m} \otimes (A_{PL})_n \otimes (A_{PL})_m $ such
that
\[
\partial^{H_m}_{k }x_{n,m}  \otimes \alpha_n
\otimes \beta_m =   x_{n+1,m} \otimes \delta^{k} \alpha_{n+1}
\otimes \beta_{m}
\]
and
\[
\partial^{V_n}_{k } x_{n,m}  \otimes \alpha_n
\otimes \beta_m =  x_{n,m+1} \otimes \alpha_n \otimes \delta^{k}
\beta_{m+1}.
\]
\end{definition}

\begin{lemma}\label{lemma 3 tot TW coincidono}
Let $V^  \blacktriangle=(V_{n,m})$ be a bisemicosimplicial differential graded vector space;
then,  the Thom-Withney construction does not depend on the order, i.e., $\tot_{TW}(\tot_{TW}^{H,\Delta} ) \cong   \tot_{TW} (\tot_{TW}^{
V,\Delta} )  \cong\tot_{TW}^\blacktriangle(V^\blacktriangle)$.
\end{lemma}

\begin{proof}
It follows from the explicit description of the Thom-Withney
construction.
\end{proof}
If  $\g^  \blacktriangle $ is a bisemicosimplicial DGLA, then, as in the semicosimplicial case, the differential graded vector space $\tot_{TW}^\blacktriangle(\g^\blacktriangle)$ inherits a structure of DGLA.

\begin{remark}
As for the semicosimplicial case, the differential graded vector spaces $\tot_{TW}^\blacktriangle(\g^\blacktriangle)$ and  $\tot^\blacktriangle (\g^\blacktriangle)$ are quasi-isomorphic.
In a forthcoming paper,    we will use the  DGLA structure of $\tot_{TW}^\blacktriangle(\g^\blacktriangle)$ and the homotopy transfer to define a canonical $L_\infty$-algebra structure
$\widetilde{\tot}^\blacktriangle (\g^\blacktriangle)$   on   $\tot^\blacktriangle (\g^\blacktriangle)$,   such that   $ \widetilde{\tot}^\blacktriangle (\g^\blacktriangle) $ and  $\tot_{TW}^\blacktriangle(\g^\blacktriangle)  $ are
  quasi-isomorphic $L_\infty$-algebra.

 \end{remark}

\subsection{Deformation functors associated with a bisemicosimplicial DGLA}\label{section. deformat bisemico DGLA}

In this section, we will describe how we can associate a deformation functor with a bisemicosimplicial DGLA.
In Section \ref{section. deformat semico DGLA}, we introduced the deformation functor $ H^1_{\rm sc}(\exp\g^\Delta)$ associated with
a semicosimplicial DGLA $\mathfrak{g}^\Delta$. Moreover, Theorem~\ref{teor.semisipl H^1=Def TOT} states that   $ H^1_{\rm sc}(\exp\g^\Delta) \simeq
\Def_{{\tot}_{TW}(\g^{\Delta})}$, whenever $H^{k}(\g_i)=0$, for all $i$ and for all $k<0$.

Next, let ${\g}^\blacktriangle $ be a  bisemicosimplicial DGLA.
In the previous section, we associated with  ${\g}^\blacktriangle $, the semicosimplicial DGLAs $\tot_{TW}^{H,\Delta} $ and $  \tot_{TW}^{V,\Delta} $. Therefore, we can naturally associate with  ${\g}^\blacktriangle $ the two deformations functors $ H^1_{\rm sc}( \exp   \tot_{TW}^{H,\Delta}) $ and $ H^1_{\rm sc}( \exp   \tot_{TW}^{V,\Delta}) $. Moreover, we associated with ${\g}^\blacktriangle $  the Thom-Whitney DGLA
$ \tot_{TW}^\blacktriangle({\g}^\blacktriangle)$ and so we can consider  its   deformation functor $  \Def_{\tot_{TW}^\blacktriangle({\g}^\blacktriangle)}$.
The following theorem explains the relation between all these functors.

\begin{theorem}\label{teor.BIsemisipl H^1=Def TOT}
Let  ${ \g}^\blacktriangle$ be  a bisemicosimplicial DGLA such that $H^{k}(\g_{i,j})=0$, for all $i,j$ and
$k<0$. Then, there exist
natural isomorphisms of deformation functors
$$
H^1_{\rm sc}( \exp   \tot_{TW}^{H,\Delta}) \cong \Def_{\tot_{TW}(\tot_{TW} ^{H,\Delta}(V^\blacktriangle))} \cong
$$
$$
\cong \Def_{\tot_{TW}^\blacktriangle({ \g}^\blacktriangle)}\cong
\Def_{\tot_{TW}(\tot_{TW}^{{V},\Delta}({ \g}^\blacktriangle))} \cong  H^1_{\rm sc}( \exp   \tot_{TW}^{V,\Delta}).
$$
\end{theorem}

\begin{proof}
The cohomological constraint of  the hypothesis implies that   $H^{k}(\tot_{TW}^{H,\Delta}({ \g}^\blacktriangle)_m)= H^{k}(\tot_{TW}^{V,\Delta}({ \g}^\blacktriangle)_n) = 0$, for all $n,m$ and for all $k<0$. Therefore, the first and last isomorphisms    follow from Theorem \ref{teor.semisipl H^1=Def TOT}. The remaining isomorphisms follow from  Lemma \ref{lemma 3 tot TW coincidono}.
\end{proof}

\begin{example}\label{example. totTW TX(Z)--TX}

(Example \ref{example. chec TY(X)--TY} revisited) Let $X$ be a smooth variety,  $Z \subset X$   a closed subscheme and 
$\mathcal{U}=\{U_i\}_i$ an affine open cover of $X$. 
 Denote by ${\chi}  : {\Theta}_X(-\log Z) \hookrightarrow
{\Theta}_X$   the inclusion of sheaves of Lie algebras.
Following  Example \ref{example. chec TY(X)--TY}, we have the
bisemicosimplicial Lie algebra $\chi^\blacktriangle: {\Theta}_X(-\log Z) (\mathcal{U}) \to {\Theta}_X (\mathcal{U})$ and so we can consider the associated DGLA $ \tot_{TW}^\blacktriangle(\chi^\blacktriangle)$.
Moreover,  as in the   the previous construction, we can associate with $\chi$ two semicosimplicial DGLAs.
The easiest way is to consider the induced morphism of DGLA
${\chi_{TW}}: \tot_{TW}({\Theta}_X(-\log Z)(\mathcal{U})) \to \tot_{TW}({\Theta}_X(\mathcal{U}))$, and view it as a  semicosimplcial DGLA   by zero extension (see Example \ref{exa morfismo come semicosim 0}), i.e.,
 \[  \chi_{TW}^\Delta :\quad
\xymatrix{ {\tot_{TW}({\Theta}_X(-\log Z)(\mathcal{U}))} \ar@<2pt>[r]^{\qquad 0}\ar@<-2pt>[r]_{\qquad {\chi_{TW}}} & { \tot_{TW}({\Theta}_X(\mathcal{U}))}
      \ar@<4pt>[r] \ar[r] \ar@<-4pt>[r] & \ 0 \cdots }.
\]
Analogously,  applying the Thom-Whitney construction firstly on the rows,  we get  the   semicosimplcial DGLA $T^\Delta$
\[
\xymatrix{    \vdots  \\
T_2=\tot_{TW}(\prod_{i<j<k}{\Theta}_X(-\log Z) (U_{ijk}) \stackrel{\chi}{\to} \prod_{i<j<k }{\Theta}_X  (U_{ijk})) \ar@<6pt>[u] \ar@<2pt>[u]
\ar@<-2pt>[u] \ar@<-6pt>[u]  \\
T_1=\tot_{TW}(\prod_{i<j}{\Theta}_X(-\log Z) (U_{ij}) \stackrel{\chi}{\to}           \prod_{i<j}{\Theta}_X (U_{ij})) \ar@<4pt>[u] \ar[u] \ar@<-4pt>[u]
 \\
T_0=\tot_{TW}(\prod_{i}{\Theta}_X(-\log Z) (U_{i}) \stackrel{\chi}{\to} \prod_{i}{\Theta}_X(-\log Z) (U_{i})).  \ar@<2pt>[u]\ar@<-2pt>[u]
  ,}
\]
In this second case, the vertical maps are the restrictions to open subsets (see Example \ref{example. chec F fascio}). The previous
Theorem \ref{teor.BIsemisipl H^1=Def TOT} implies that there exist isomorphisms of deformation functors
$$
\Def_{\chi_{TW}} \cong
\Def_{\tot_{TW}^\blacktriangle(\chi^\blacktriangle)}\cong H^1_{\rm sc}( \exp      T^\Delta   ).
$$
We recall that the functor $\Def_{\chi_{TW}}$ is isomorphic to  $ H^1_{\rm sc}( \exp      \chi_{TW}^\Delta   )$ (see Example \ref{example.funtore morfismo DGLA}).
More explicitly, since $\chi$ is injective, for all $A \in \Art$, the set $\Def_{\chi_{TW}}(A)$ is given by
\[
\Def_{\chi_{TW}}(A) = \frac{\MC_{\chi_{TW}}(A)}{\sim},
\]
where
\begin{align*}
 \MC_{\chi_{TW}}(A)= &
 \{a\in  {\tot_{TW}({\Theta}_X(\mathcal{U}))}^0\otimes
\mathfrak{m}_A)\ |\\
&
\  e^{-a}\ast0 \in {\tot_{TW}({\Theta}_X(-\log Z)(\mathcal{U})}) ^1\otimes\mathfrak{m}_A  \},\\
\end{align*}
and $e^a\sim e^{a'} $ if and only if there exist $b \in
{\tot_{TW}({\Theta}_X(-\log Z)(\mathcal{U}))} ^0\otimes\mathfrak{m}_A,$ such that     $e^{a'} =   e^a e^{- b } $.

\end{example}

\section{Application: Deformations of subvarieties in a fixed smooth variety}\label{section deformazioni Z in X}

Let $X$ be a smooth variety, defined over an algebraically closed
field $\K$ of characteristic 0, and  $Z \subset X$   a  closed subscheme. We recall the definition of infinitesimal deformations of $Z$ in $X$ fixed, full details can be found in \cite{Sernesi}.
\begin{definition}
Let $A\in \Art$. An infinitesimal deformation of $Z$ in $X$
over $\Spec(A)$ is a cartisian  diagram
\begin{center}
$\xymatrix{Z \ar[rr]^i\ar[d] &  & Z_A   \ar[d]^{\pi}
\subset X \times \Spec(A)  \\
           \Spec(\K) \ar[rr]^{a} & & \Spec(A),  &  \\ }$
\end{center}

where $\pi$ is a    flat map  induced by the projection from $X
\times \Spec(A)$ to $\Spec(A)$.
The associated  infinitesimal deformation functor is
$$
\Hilb_{X}^Z : \Art \to \Set,
$$
such that
$$
\Hilb_{X}^Z(A)= \{
 \mbox{infinitesimal   deformations of $Z$ in $X$ over
 $\Spec(A)$} \}.
$$
\end{definition}
Moreover, we can  define the sub-functor
$$
{\Hilb'}_{X}^Z : \Art \to \Set,
$$
where
$$
{\Hilb'}_{X}^Z(A)= \{
 \mbox{locally trivial infinitesimal   deformations of $Z$ in $X$ over
 $\Spec(A)$}  \}.
$$
We recall that, an infinitesimal deformation $Z_A$ of $Z$ in $X$
over $\Spec(A)$ is called locally trivial if, for every point $z
\in Z$, there exists an open neighbourhood $U_z \subset Z$ such
that
\begin{center}
$\xymatrix{U_z \ar[r]^i\ar[d] &   {Z_A}_{|U_z}   \ar[d]^{\pi}   \\
           \Spec(\K) \ar[r]^{a} & \Spec(A),  &  \\ }$
\end{center}
is a trivial deformation of $U_z$.
Whenever $Z\subset X$ is smooth, then every deformation of $Z$ in $X$ is locally trivial and so ${\Hilb}_{X}^Z={\Hilb'}_{X}^Z$.

\smallskip

Next, following  Examples \ref{example. chec TY(X)--TY} and \ref{example. totTW TX(Z)--TX}, denote by  ${\chi}  : {\Theta}_X(-\log Z) \hookrightarrow
{\Theta}_X$,  the   inclusion of sheaves of Lie algebras, and by $\chi^\blacktriangle: {\Theta}_X(-\log Z) (\mathcal{U}) \to {\Theta}_X (\mathcal{U})$, the associated bisemicosimplicial Lie algebra, where $\mathcal{U}$ is an affine open cover of $X$.

\begin{theorem}
Let $X$ be a smooth variety, defined over an algebraically closed
field $\K$ of characteristic 0, and  $Z \subset X$   a closed subscheme. Then,
there exists an isomorphism of functors
$\Def_{\tot_{TW}^\blacktriangle(\chi^\blacktriangle)} \cong
{\Hilb'}_{X}^Z$. In particular, if $Z\subset X$ is smooth, then
$\Def_{\tot_{TW}^\blacktriangle(\chi^\blacktriangle)} \cong
{\Hilb}_{X}^Z$.
\end{theorem}

\begin{proof}
For $\K= \mathbb{C} $ and $Z$ smooth, this theorem was already proved in \cite[Theorem 5.2]{ManettiSemireg}, without the use of  semicosimplicial language.

Let $\mathcal{U}=\{U_i\}_{i \in I}$ be an affine open cover of $X$ and $\mathcal{V}=\{V_i=U_i\cap Z\}_{i \in I}$ the induced one on $Z$.
Let  $Z_A$ be a  locally trivial deformation of $Z$ in $X$
over $\Spec(A)$.
Then,   $Z_A$ is obtained by gluing the trivial deformations  $V_i\otimes A$  in $U_i \otimes A$ along the double intersections $V_{ij}\otimes A$, such that the induced deformation of $X$ is trivial.
Therefore, it is determined by a sequence   $\{\theta_{ij}\}_{i<j}$ of automorphisms of the sheaves of $A$-algebras
\[
\xymatrix{ &\mathcal{O}_Z(V_{ij}) & \\
\mathcal{O}_Z(V_{ij})\otimes A \ar[rr]^{\theta_{ij}}_{\simeq}\ar[ur] & &
 \mathcal{O}_Z(V_{ij})\otimes A\ar[ul] \\
           & A,\ar[ru]\ar[lu] &   \\ }
\]
satisfying the cocycle condition
\begin{equation}\label{equa.cociclo auto Vij}
 \theta _{jk}  \theta _{ik}^{-1}
   \theta _{ij}=\Id_{\mathcal{O}_Z(V_{ijk})\otimes A}, \qquad
   \forall \ i<j<k \in I,
\end{equation}
and such that there exist automorphisms  $\alpha_i$ of $\mathcal{O}_X(U_i)\otimes A$ satisfying
\begin{equation}\label{equa.ident su  Uij}
{\theta_{ij}}= {\alpha_i}^{-1}
{\alpha_j} , \qquad \forall i<j.
\end{equation}
Note that  Equation (\ref{equa.ident su  Uij}) implies  (\ref{equa.cociclo auto Vij}).
Since we are in characteristic zero, we  can take the logarithms
and write $\theta_{ij}=e^{d_{ij}}$,  for some
$d_{ij}\in\Theta_X(-\log Z)(U_{ij})\otimes\mathfrak{m}_A$, and $\alpha_i=e^{a_i}$, with $a_i \in \Theta_X(U_i)\otimes\mathfrak{m}_A$.

Therefore,  a    locally trivial deformation of $Z$ in $X$
over $\Spec(A)$ is equivalent to the datum of a sequence $\{ a_i\}_i  \in \prod_i \Theta_X(U_i)\otimes\mathfrak{m}_A$, such that
\[ e^{-a_i}e^{a_j} \in
 \exp(\Theta_X(-\log Z) (U_{ij})\otimes\mathfrak{m}_A) , \qquad \forall \ i<j \in I.
\]

As regards the equivalence relation,  let  $Z_A$  and $Z_A'$ be two deformations of $Z$ in $X$  over  $\Spec(A)$. Denote by $\theta_{ij}=e^{d_{ij}}=e^{-a_i}e^{a_j}$ and $\theta'_{ij}=e^{d'_{ij}}=e^{-a'_i}e^{a'_j}$ the data associated  with $Z_A$  and $Z_A'$, respectively.
The deformations  $Z_A$  and $Z_A'$ are isomorphic if,
 for every
$i$, there exists an automorphism $\beta_i$ of $\mathcal{O}_Z(V_i)\otimes A$,
such that $\theta_{ij}= {\beta_i}^{-1}
{\theta'_{ij}} \beta_j$, for every $i<j$, and satisfying the compatibility relation ${\alpha_i'} \beta_i =\alpha_i$.
Taking again logarithms, an isomorphism between $Z_A$  and $Z_A'$  is equivalent to the existence of a sequence   $\{b_i\}_i \in \prod_i\Theta_X(-\log  Z)(U_i)\otimes\mathfrak{m}_A$, such that $e^{a_{i}} =e^{a'_i}e^{ b_i }$.

Next, from the DGLA point of view, we showed in Example   \ref{example. totTW TX(Z)--TX}, that
$
\Def_{\tot_{TW}^\blacktriangle(\chi^\blacktriangle)}\cong H^1_{\rm sc}( \exp      \chi^\blacktriangle     ) \cong \Def_ {\chi_{TW}}$. Therefore,    it is enough to prove that ${\Hilb'}_{X}^Z \cong \Def_ {\chi_{TW}}$, with
${\chi_{TW}}: \tot_{TW}({\Theta}_X(-\log Z)(\mathcal{U})) \hookrightarrow  \tot_{TW}({\Theta}_X(\mathcal{U}))$; and
it follows from the explicit description of $\Def_ {\chi_{TW}}$. Indeed,  $\MC_ {\chi_{TW}}(A)$ is the set of all $a\in
  {\tot_{TW}({\Theta}_X(\mathcal{U}))}^0\otimes
\mathfrak{m}_A  $, such that $
e^{-a}\ast0 \in {\tot_{TW}({\Theta}_X(-\log Z)(\mathcal{U})}) ^1\otimes\mathfrak{m}_A$, i.e., $a=\{ a_i\}_i \in \prod_{i}{\Theta}_X  (U_{i}) \otimes \mathfrak{m}_A $, such that $e^{-a_i}e^{a_j} \in
 \exp(\Theta_X(-\log Z) (U_{ij})\otimes\mathfrak{m}_A) $. Moreover,
$a\sim a'$ if and only is   there exist $b \in
{\tot_{TW} ({\Theta}_X(-\log Z)(\mathcal{U}))} ^0\otimes\mathfrak{m}_A,$ such that     $e^{a'} =   e^a e^{- b } $, i.e., $b=\{b_i\}_i \in \prod_i\Theta_X(-\log  Z)(U_i)\otimes\mathfrak{m}_A$ such that $e^{a_{i}} =e^{a'_i}e^{ b_i}$.

\end{proof}

\begin{remark}
In a forthcoming paper,    we will use this theorem to study the obstructions to the  deformations of $Z$ in $X$, via the semiregularity map.
 \end{remark}

\end{document}